\documentclass[12pt]{amsart}

\usepackage{amsmath,amsfonts,amsthm,amsopn,cite,mathrsfs}
\usepackage{epsfig,verbatim}
\usepackage{subfigure}

\newcommand{\fif}{if and only if}
\newcommand{\tfs}{time-frequency shift}

\newtheorem{tm}{Theorem}[section]    
\newtheorem{lemma}[tm]{Lemma}
\newtheorem{prop}[tm]{Proposition}

\theoremstyle{definition}

\newcommand{\beqa}{\begin{eqnarray*}}
\newcommand{\eeqa}{\end{eqnarray*}}

\newcommand{\field}[1]{\mathbb{#1}}
\newcommand{\bR}{\field{R}}        
\newcommand{\bN}{\field{N}}        
\newcommand{\bZ}{\field{Z}}        
\newcommand{\bC}{\field{C}}        
\newcommand{\bQ}{\field{Q}}        
\def\cF{\mathcal{ F}}              

\def\cG{\mathcal{ G}}

\def\<{\left<}
\def\>{\right>}

\def\inv{^{-1}}

\def\mv1{M_v^1}

\newcommand{\gab}{\cG (g,\alpha, \beta)}
\newcommand{\tp}{totally positive}
\newcommand{\ltwo}{L^2(\bR )}

\begin{document}
\begin{abstract}
We show that for an arbitrary totally positive function $g\in L^1(\bR )$ and 
 $\alpha \beta$ rational,  the Gabor family $\{e^{2\pi i \beta l
  t} g(t-\alpha k): k,l \in \bZ \}$ is a frame for $L^2(\mathbb{R})$, if and only if $\alpha
\beta <1$.  
\end{abstract}

\title{ Totally Positive Functions and Gabor Frames over Rational Lattices}
\author{Karlheinz Gr\"ochenig}
\address{Faculty of Mathematics \\
University of Vienna \\
Oskar-Morgenstern-Platz 1 \\
A-1090 Vienna, Austria}
\email{karlheinz.groechenig@univie.ac.at}
\date{}
\thanks{K.\ G.\ was
  supported in part by the  project P31887-N32  of the
Austrian Science Fund (FWF)}
\maketitle

\section{Introduction}

Ever since J.\ von Neumann and D.\ Gabor claimed that the family of
\tfs s (phase-space shifts) $\{ e^{2\pi i lt} e^{-\pi (t-k)^2}: k,l\in
\bZ \}$ spans $L^2(\bR ) $, the spanning properties of \tfs s have
sparked the curiosity of mathematicians, physicists, and
engineers. Modern versions of this problem replace the Gaussian by
arbitrary generators and the \tfs\ over the lattice $\bZ ^2$ by
general lattices or even non-uniform sets.

To be specific, the  \tfs\ of $g\in L^2(\bR )$   by $(x,\xi ) \in \bR ^2$ is
defined as
\begin{equation}
  \label{eq:1s}
  M_\xi T_xf(t) = e^{2\pi i \xi t} g(t-x) \, .
\end{equation}
Given  lattice parameters $\alpha ,\beta >0$ the set 
$\cG (g,\alpha,\beta )  = \{ M_{\beta l} T_{\alpha k} g: k,l \in \bZ
\}$ is called  Gabor family  over the lattice $\alpha \bZ
\times \beta \bZ $ with window $g$. Gabor analysis studies the question under which
conditions on $\alpha, \beta $,  and $g$   there exist constants  $A,B>0$, such that 
\begin{equation}
  \label{eq:7}
A \|f\|_2^2  \leq  \sum _{k,l\in \bZ } | \langle f, 
M_{\beta l}T_{\alpha k} g\rangle |^2 \leq B \|f\|_2^2    \qquad \forall f\in L^2(\bR
)  \, .
\end{equation}
If \eqref{eq:7} holds, $\cG (g,\alpha , \beta )$ is called a Gabor
frame.  The set
$$
\cF (g) = \{ (\alpha,\beta ) \in \bR _+^2: \cG (g,\alpha ,\beta ) \,
\text{ is a frame }\}
$$
is called the frame set of $g$. Knowing $\cF (g)$ amounts to a complete
characterization of all (rectangular) lattices $\alpha \bZ \times
\beta \bZ $ that generate a Gabor frame with a given  window $g$. 

Our understanding of the frame set is still very limited. The
prototypical result is due to Lyubarski~\cite{lyub92} and
Seip-Wallsten~\cite{seip92,seip-wallsten} and determines the frame set
of the Gaussian to be $\cF (e^{-\pi t^2}) = \{ (\alpha, \beta ) \in
\bR _+^2: \alpha \beta < 1\}$. This is usually formulated by saying
that $\cG (e^{-\pi t^2}, \alpha , \beta )$ is a frame \fif\ $\alpha
\beta <1$. Subsequently, further examples of frame sets were found by
Janssen and Strohmer~\cite{janssen96,janssen02b,JS02}. It was noted
in~\cite{GS13} that all examples known  until  2010 were based on a  \tp\
window $g$. This observation led to a more systematic approach to
study the frame set for \tp\ windows.

To explain the state-of-the-art we recall that 
a measurable function $g $ on $\bR $  is \tp , if for every
$n\in \bN$ and every two sets of
increasing numbers $x_1 < x_2 <\dots < x_n$ and $y_1< y_2 < \dots <
y_n$ the matrix $\big( g (x_j - y_k) \big) _{j,k=1, \dots ,n}$
has non-negative determinant:
\begin{equation}
  \label{eq:1}
  \det ( g (x_j - y_k) \Big) _{j,k=1, \dots ,n} \geq 0 \, .
\end{equation}
If in addition $g $ is integrable, then $g $ possesses
already exponential decay~\cite[Lemma~2]{sch51}. Often an integrable \tp\ function is called
a  Polya frequency function, but we will not make this distinction
here.

By a deep theorem of Schoenberg~\cite{sch51} the Fourier transform of an
integrable \tp\ function $g$ possesses the factorization
\begin{equation}
  \label{eq:t1}
  \hat{g  } (\xi ) = c e^{-\gamma \xi ^2} e^{2\pi i \nu \xi } \prod
_{j=1}^N (1+2\pi i \nu _j \xi )\inv e^{-2\pi i \nu _j \xi }  
\end{equation}
with $c>0$, $\nu , \nu _j \in \bR $, $\gamma \geq 0$, $N\in \bN \cup \{
\infty \}$ and $0< \gamma + \sum _j \nu _j ^2 <\infty $.

Among the \tp\ functions are the
Gaussian $e^{-\gamma t^2}$ for $\gamma >0$ and  the one-sided exponential functions 
$\eta _\gamma (t) = e^{-\gamma t}
\chi _{\bR ^+} (\gamma t)$ with $\gamma \in \bR , \gamma \neq 0$.  

For the subclass of \tp\ functions, for which the factorization~\eqref{eq:t1}
is a
finite product, the results of~\cite{GS13,GRS18} can be summarized as
follows:
\begin{tm} \label{old} 
  Let $g\in \ltwo $ be a \tp\ function with a finite factorization
  \eqref{eq:t1} with  $N\in
  \bN $ other than the one-sided exponentials  $\eta _\gamma $, then
  $$
\cF (g) = \{ (\alpha, \beta ) \in
\bR _+^2: \alpha \beta < 1\}   \, .
$$
\end{tm}
In other words, the lattice $\alpha \bZ \times \beta \bZ $ generates a
frame with \tp\ window $g$, \fif\ $\alpha \beta <1$.

Remarkly the proof of Theorem~\ref{old} required different methods for
the cases $\gamma =0$ (no Gaussian factor) and $\gamma >0$ (with Gaussian
factor). In the former case total positivity and the
Schoenberg-Whitney conditions were used, in the latter case only the
factorization was used in conjunction with complex analysis.

The frame set  of the one-sided exponential $\eta _\gamma $ is
slightly different, since $\eta _\gamma $  is not smooth enough and
the Balian-Low theorem does not apply~\cite{heil07}. In this case the
frame set is $\cF ( \eta _\gamma ) =  \{ (\alpha, \beta ) \in
\bR _+^2: \alpha \beta \leq 1\}$, as was already known by Janssen~\cite{janssen96}.

Theorem~\ref{old} suggests that the characterization of Gabor frames
holds for \emph{all} \tp\ generators.  Our goal is to prove the
following theorem which is a step towards the full conjecture
formulated in ~\cite{Gr14}. 

\begin{tm}\label{new}
  Assume that $g$ is totally positive other than  a one-sided exponential function
  $\eta _\gamma $. Further assume that  $\alpha \beta $ is
  rational. Then  $\gab $ is a frame for $L^2(\bR )$, \fif\ $\alpha
  \beta <1$.
\end{tm}
Theorem~\ref{new} implies that  the frame set $\cF (g)$ is an \emph{open,
    dense } subset of  $\{ (\alpha, \beta ) \in
 \bR _+^2: \alpha \beta < 1\} $. This follows from a deformation
 result  of Feichtinger and Kaiblinger~\cite{FK04}.

Since every \tp\ function is a limit of a \tp\ with a finite
factorization, one would   hope that some kind of  limiting procedure
would work. However, so far all our ``proofs'' contained a
gap. Except for the example of the  hyperbolic secant $(e^{at}+ e^{-at})\inv
$~\cite{JS02}, Theorem~\ref{new} is the first general statement about
Gabor frames with a \tp\ generator with an  infinite factorization~\eqref{eq:t1}. 

 In
this paper we present a new approach that is independent of
Schoenberg's factorization~\eqref{eq:t1} and uses only the definition of
total positivity. The new proof idea comes from   a deep theorem of \cite{BFP82} about
invertibility of infinite \tp\ matrices. We also use methods from
spectral invariance of matrix algebras and the important fact that the Zak
transform of a \tp\ function has only one zero.  For
rational lattices this method  leads also  to a  new proof of Theorem~\ref{old}.

To finish the discussion, we mention the superb work ~\cite{BKL22},
which 
determines the frame set of
rational functions (Hurwitz functions) $g(t) = \sum _{j=1}^N
\frac{a_j}{t-iw_j}$ for $w_j\in \bC $. If $a_j >0$ and $w_j >0$, then
$\cF (g) = \{ (\alpha, \beta ) \in
 \bR _+^2: \alpha \beta \leq 1\} $, for general coefficients the frame
 set contains  $\{ (\alpha, \beta ) \in
 \bR _+^2: \alpha \beta \leq 1\} \cap \bQ ^c$. Interestingly and in 
 contrast to our work, the authors show that all irrational lattices
 belong to the frame set!

Other than this the complete frame set has been determined only for
the characteristic function of an interval  $\chi
_{[0,1]}$~\cite{sunabc} and the Haar function~\cite{DZ22}. In contrast
to the results for \tp\ windows and rational functions, the frame set
is extremely complicated and depends on subtle number theoretic
properties of the product $\alpha \beta $.

The case of irrational lattices remains open. There is some evidence
that generically a lattice $\alpha \bZ \times \beta \bZ $ with
irrational $\alpha \beta $ generates a Gabor frame. Important results
pointing in this direction are contained in~\cite{BKL22} and
in~\cite{DZ22}. In each case an underlying dynamical system and
some ergodicity are at work. Currently we do not know how this works
for \tp\ windows.

Of course, much more is known and relevant about Gabor frames. As we
are interested only in the complete characterization of Gabor frames
for a given window, we refer to the surveys~\cite{Gr10,GK19,heil07} and
monographs~\cite{book,CR20} for the general theory of Gabor frames.

The paper is organized as follows: In Section~2 we review the tools
underlying Gabor analysis. These are Zak transform methods, Banach
algebra methods concerning spectral invariance, and general 
characterizations of Gabor frames.  All these characterizations  amount to showing that
some (infinite) matrix is invertible. In our approach the matrix  $G$ under consideration
will  a sub-matrix of the so-called pre-Gramian. This is a technical
novelty. In Section~3 we show that this matrix is onto $\ell ^1
(\bZ )$ by using heavily the properties of spectral invariance and a
theorem from~\cite{BFP82}. In
Section~4 we will show that $G$ is also one-to-one. It is only here
that we will use the rationality of the lattice. Ultimately the proof
boils down to the fact that a square matrix is invertible, \fif\ it is
onto. 

\vspace{3mm}

\textbf{Acknowledgement.} I am deeply indebted to Jose Luis Romero and
Joachim St\"ockler for comments on earlier versions of this
manuscript. 

\section{Tools}

\subsection{Reduction.}  Let $g_\beta (x) = \beta ^{-1/2}
g(x/\beta )$. Then $\cG (g, \alpha \bZ \times \beta \bZ ) $ is a
frame, \fif\ $\cG (g_\beta, \alpha \beta \bZ \times  \bZ ) $ is a
frame. It is immediate from \eqref{eq:1} that the set of \tp\ functions is invariant under
dilations. Thus, 
if $g$ is totally positive, then so is $g_\beta
$. We may therefore always replace $g$ by $g_\beta $, $\alpha $ by
$\alpha \beta $ and $\beta $ by $1$ and  thus  assume without
generality that 
$\beta =1$. 

We note that we may always assume that $\alpha \beta <1 $ and, after
the reduction, that $\alpha <1,\beta =1$ because of the density theorem for Gabor
frames~\cite{heil07}. 

\subsection{Spectral invariance.}
Spectral invariance refers to the phenomenon that an invertible, infinite matrix with
sufficiently strong off-diagonal decay  possesses an inverse with the same
off-diagonal decay. 
This topic has a long  history with many contributions and
variations. See ~\cite{ABK08,Bas90,Sjo95,Sun07a,Sun08b,Tes10} for a sample of
contributions and the survey~\cite{Gr10}. The following  version can be attributed to
Shin-Sun~\cite{Sun08b} and Tessera~\cite[Thm.~1.8]{Tes10}. 

\begin{prop}[Stability Theorem]  \label{stab}
   Let $G = (G_{kl})_{k,l\in \bZ }$ be a bi-infinite matrix  with off-diagonal decay
  $$
  |G_{kl}| \leq C (1+|k-l|)^{-\sigma} \qquad k,l \in \bZ \, ,
  $$
  for some $\sigma >1$.

  (i) \textrm{Spectral invariance:} If $G$ is invertible on some $\ell ^{p_0}(\bZ )$, 
  then $G$ is invertible  on all $\ell ^p(\bZ )$ for $1\leq p\leq \infty
  $. 

  (ii) \textrm{Stability:} If $G$ is stable on some $\ell ^{p_0}(\bZ )$, i.e., if $G$ satisfies
  $$
  \|Gc\|_{p_0} \geq A \|c\|_{p_0} \qquad \text{ for all } c\in \ell
  ^{p_0}(\bZ ),
  $$
  then $G$ is stable on all $\ell ^p(\bZ )$ for $1\leq p\leq \infty
  $.
\end{prop}

We will need a slight variation of the stability (ii) for surjective
maps.

\begin{prop} \label{surjlem}
  Let $G = (G_{kl})_{k,l\in \bZ }$ with off-diagonal decay
  $$
  |G_{kl}| \leq C (1+|k-l|)^{-\sigma} \qquad k,l \in \bZ \, ,
  $$
  for some $\sigma >1$. If $G$ maps $\ell ^\infty (\bZ )$ onto $\ell
  ^\infty (\bZ )$, then $G$ also maps   $\ell ^1 (\bZ )$ onto $\ell
  ^1 (\bZ )$
   \end{prop}
   \begin{proof}
     Note that the off-diagonal decay implies that $G$ defines a
     bounded operator on all $\ell ^p(\bZ ), 1\leq p\leq \infty $.
     
     By the closed range theorem~\cite[p.~102]{rudin73}, the adjoint operator
     (matrix) on the pre-dual $G^*: \ell ^1(\bZ ) \to \ell
     ^1(\bZ) $ is one-to-one with closed range in $\ell ^1(\bZ )$.
     Consequently, $G^*$ satisfies an inequality of the form
     \begin{equation}
       \label{eq:2}
     \|G^* c\|_1 \geq A \|c\|_1 \qquad \text{ for all } c\in \ell ^1(\bZ )  
     \end{equation}
      with a positive constant $A>0$.

The stability part of Proposition~\ref{stab} then implies that \eqref{eq:2}
     holds for all $p$-norms, in particular, for some $A'>0$ we have 
     $$
     \|G^* c\|_\infty \geq A' \|c\|_\infty  \qquad c\in \ell ^\infty
     (\bZ ) \, .
     $$
Since $G^*$ is one-to-one with norm-closed range in $\ell
     ^\infty (\bZ ) $, we conclude that 
     $G:  \ell ^1(\bZ ) \to \ell
     ^1(\bZ) $ must be onto $\ell ^1(\bZ )$. 
   \end{proof}

\subsection{Characterizations of Gabor frames.}

We use the following characterization of Gabor frames that goes back
to Janssen~\cite{janssen95} and the
duality theory of Gabor frames~\cite{ron-shen97}. This is one of the
most successful criteria to verify that a Gabor system over a
rectangular lattice generates a frame and has been applied many
times.  We refer to~\cite{GK19} for a survey of the known characterizations of Gabor
frames and some of their uses.

\begin{tm}\label{thm:RSLambda}
Assume that  $g: \bR \to \bR $ is continuous and satisfies the
condition  $\sum
_{k\in \bZ } \sup _{x\in [0,1]} |g(x+k)| < \infty $. 
Then the following are equivalent.
\begin{itemize}
\item[(i)]
The family $\cG(g, \alpha \bZ \times \bZ)$  is a frame for $L^2(\bR)$.

\item[(ii)] 
  For every $x \in [0,1)$, there
exist $A_x,B_x>0$
such that
\begin{equation}
\label{eq:samplingset2}
A_x\|c\|_2^2 \leq \sum_{j\in \bZ}
\Big| \sum_{k \in \bZ} c_k g(x+ \alpha j-k) \Big|^2
  \leq B_x \|c\|_2^2
  \qquad \mbox{for all } c  \in \ell^2(\bZ ).
\end{equation}

\item[(iii)]  There exist $A,B>0$ such that, for all $ x\in [0,1)$,
\begin{equation*}
A\|c\|_2^2 \leq \sum_{j\in \bZ }
\Big| \sum_{k \in \bZ} c_k g(x+\alpha j -k) \Big|^2
  \leq B \|c\|_2^2
    \qquad\mbox{for all } c  \in \ell^2(\bZ ).
\end{equation*}
\end{itemize}
\end{tm}

The matrix $P(x) = \big(g(x+\alpha j - k)\big)_{j,k\in \bZ} $ is often 
called the pre-Gramian of the Gabor family.

Using spectral invariance, we deduce the following sufficient
condition for Gabor frames. 

\begin{prop}\label{charcis}
  Let $g $ 
  be a function on $\bR $ with polynomial
  decay $|g(t)| \leq C (1+|t|)^{-\sigma }$ for some $\sigma >1$.  
  Assume that for every $x\in \bR $ the set  $x+\alpha \bZ $  contains a subset $\{k+\delta
  _k: k\in \bZ \}$ with the perturbation $\delta _k = \delta _k(x)$
  depending on $x$,  such that the matrix $G$ with entries $G_{kl} =
  g(k+\delta _k - l)$ is invertible on $\ell ^\infty (\bZ )$. Then
  $\cG (g, \alpha ,1)$ is a Gabor frame. 
\end{prop}
\begin{proof}
  The decay of $g$ implies that $G$ has sufficient off-diagonal decay
  to apply  Proposition~\ref{stab}(i). Thus  spectral invariance
  asserts that the invertibility of $G$ on $\ell
  ^\infty (\bZ )$ implies that $G$ is also invertible on $\ell ^2(\bZ
  )$. Since $x+ \alpha \bZ \supseteq \{k+\delta _k: k\in \bZ \}$, we
  have 
  \begin{align*}
    \sum _{j\in \bZ } \big| \sum _{l\in \bZ } c_l g(x+\alpha j -l)|^2
&    \geq \sum _{k\in \bZ } \big| \sum _{l\in \bZ } c_l g(k+\delta _k -l)|^2\\
&= \|Gc\|_2^2 \geq A_x \|c\|_2^2 \, .     
  \end{align*}
  The upper bound in~\eqref{eq:samplingset2}  always holds due to the
  decay of $g$.  This holds for all $x$, therefore the condition (ii) of
Theorem~\ref{thm:RSLambda} implies that $\cG (g,\alpha, 1)$ is a frame. 
\end{proof}

Proposition~\ref{charcis} says that a suitable submatrix of the
pre-Gramian $P(x)$ is invertible. Proposition~\ref{charcis}  is related to the characterizations via
sampling in shift-invariant spaces that figure prominently
in~\cite{GRS18}. (In an equivalent formulation,  every set
$x+\alpha \bZ $ contains a perturbation of $\bZ $ that is
interpolating in the shift-invariant space associated to $g$.)

To prove the main theorem (Theorem~\ref{new}), we will verify the conditions of
Proposition~\ref{charcis}. 

\subsection{The Zak transform.} We recall the definition of the Zak transform  $Zg(x,\xi)=
\sum_{k\in\bZ} g(x-k)e^{2\pi i  k\xi}$ of a function  $g\in L^1(\bR
)$. More generally, the Zak transform with period $p>0$ is defined by 
$$
Z_pg(x,\xi)=
\sum_{k\in\bZ} g(x-pk)e^{2\pi i  pk\xi} \, .
$$

The Zak transform is $\tfrac{1}{p}$-periodic in the second variable
and quasi-periodic in the first variable, i.e., for all $x,\xi \in \bR $
$$
Z_pf\big(x+pk, \xi + \tfrac{l}{p}\big) = e^{2\pi i pk \xi } Z_pg(x,\xi ) \, .
$$


Zak transforms of \tp\ function possess only few zeros. This important
property was discovered only recently~\cite{KS14,Klo15,UV17} and is crucial for our
arguments.
\begin{prop}
The  Zak transform of a \tp\ function $g$  possesses a
unique zero in $[0,1)^2$ that is assumed  at  $(x_0,1/2)$ for some $x_0\in [0,1)$.
\end{prop}

\section{Surjectivity.} 
For the proof of Theorem~\ref{new} we verify the  sufficient condition of
Proposition~\ref{charcis} for a \tp\ function $g$. First we prove the
surjectivity. Notice that the  arguments in this section  hold for arbitrary values of
$\alpha \in (0,1)$.

\begin{prop} \label{surj}
Let $g$ be an arbitrary \tp\ function in $L^1(\bR )$ and assume that
its Zak transform has its unique zero at $(x_0,1/2)$.    Let $\epsilon
\in (0,1/2)$, $M\in \bZ $,    and  $(\delta _k) _{k\in \bZ } \subseteq [x_0+M-1+\epsilon ,
  x_0+M-\epsilon ] $ be a sequence of perturbations. Then the matrix $G$
  with entries 
  $G_{kl} = g(k+\delta _k-l), k,l\in \bZ $,  maps $\ell ^\infty (\bZ
  )$ \emph{onto} $\ell ^\infty (\bZ   )$ and $\ell ^1(\bZ )$
  \emph{onto} $\ell ^1(\bZ )$. 
\end{prop}

\begin{proof}
For the proof we use the following  fundamental result of 
  de Boor,  Friedland, and Pinkus~\cite[Cor.~1]{BFP82}. \emph{Assume that  $G$ is an  \tp\
  matrix that is bounded on $\ell ^\infty ( \bZ )$\footnote{The
    original results are formulated for sign-regular matrices and more
    general index sets.}. If the range of
  $G$ contains a uniformly alternating vector, then $G$ is onto $\ell
  ^\infty (\bZ )$. } This result is the key tool in our proofs. Related results on the invertibility of infinite
\tp\ matrices can be found in~\cite{CDMS81,Boor82,Dem82}.

Uniformly alternating means that we need to
  find a sequence $c\in \ell ^\infty (\bZ )$ such that $u=Gc$ satisfies
  $u(k) u(k+1) < 0$ for all $k\in \bZ $ and $\inf _{k\in \bZ } |u(k)| >0$.

  To accomplish this, we take $c_l=(-1)^l$ and
  compute $(Gc)_k$. This is
  \begin{align*}
    (Gc)_k &= \sum _{l\in \bZ } g(k+\delta _k - l ) (-1)^l \\
           &= Zg(k+\delta _k,\tfrac{1}{2}) \\
    &= (-1)^k Zg(\delta _k, \tfrac{1}{2}) \, ,
  \end{align*}
where the last identity comes from the quasi-periodicity of $Zg$. 
  Since $Zg(x,\tfrac{1}{2})$ is real-valued and  by assumption  $Zg(x,\tfrac{1}{2}) $ does not have any zeros on
the interval $[x_0-1+\epsilon ,   x_0-\epsilon ]$, the numbers
$Zg(\delta _k, \tfrac{1}{2})$ all have the same sign and $|Zg(\delta
_k, \tfrac{1}{2})| \geq \nu >0$ for all $k$. Consequently $u=Gc$ is
uniformly  alternating.

To apply the result of ~\cite{BFP82}, we observe that the matrix $G$  with
entries $g(k+\delta _k - l)$ is indeed \tp\ in the sense that  all
minors of $G$ of arbitrary dimension have a non-negative
determinant. But this follows from the definition~\eqref{eq:1} of total positivity.   

Thus  by ~\cite[Cor.~1]{BFP82} the matrix $G$ is onto
$\ell ^\infty (\bZ )$.   By Proposition~\ref{surjlem} $G$ is also onto $\ell
^1(\bZ )$. 
\end{proof}

To apply the above proposition, we need to select a suitable subset of
$x+\alpha \bZ $. This can always be done with the following lemma.

\begin{lemma} \label{lem7}
  Let $\alpha <1$, $\epsilon  < (1-\alpha )/2$, $M\in \bZ$,  $x_0 \in [0,1]$
  fixed,  and $x\in [0,1]$ arbitrary.

  (i) Then $x+ \alpha \bZ $ contains a subset $\{k+\delta
  _k: k\in \bZ \}$ with $\delta _k \in [x_0+M-1+\epsilon , x_0 +M-\epsilon
  ]$ for all $k\in \bZ $.

  (ii) If $\alpha $ is rational, $\alpha = p/q$, then the perturbation
  $\delta _k$ can be chosen to be $p$-periodic, $\delta _{k+np} =
  \delta _k$ for all $k,n\in \bZ $. 
\end{lemma}
\begin{proof}
(i)   For arbitrary $k\in \bZ $ the interval $[k+ x_0-1+\epsilon , k+ x_0
  -\epsilon ]$ has length $1-2\epsilon >\alpha $ and thus contains at
  least one point of $x+ \alpha \bZ $, which we can write as $k+\delta
  _k$. 

  (ii) Let $\alpha = p/q \in \bZ $. For $l= 0,\dots , p-1$ choose an
  integer $j_l \in \{ 0, \dots, q-1\}$ such that
  $$
  x+\tfrac{p}{q}j_l \in [l+x_0 -1+\epsilon , l+x_0 -\epsilon ]
  $$
  and set
  $$
  \delta _l =   x+\tfrac{p}{q}j_l -l \, .
  $$
  Clearly, $0\leq j_0 < j_1 < \dots < j_{p-1} < q$. Now let $m= l +
  np$ with $l \in [0,p-1] \cap \bZ $ and $n\in \bZ$. Then
  $x+ \tfrac{p}{q}(j_l + nq) = x+\tfrac{p}{q}j_l +np  \in [l+np +x_0
  -1+\epsilon , l+np + x_0 -\epsilon ] = 
  [m+x_0 -1+\epsilon , m+x_0 -\epsilon ]$,
  and $\delta _m = x_0 + \tfrac{p}{q}(j_l +nq) - m = x_0 +
  \tfrac{p}{q}(j_l +nq) - (l+np) = \delta _l $. Thus $\delta _k $ is
  $p$-periodic.  
\end{proof}

\section{Injectivity.} For the surjectivity of $G$ we have 
     used the fact that $g$ is \tp , but nothing about the parameter
     $\alpha $.  Now we impose in addition that  $\alpha = p/q$ is
     rational.  In this case the action of $G$ is unitarily equivalent
     to a $p\times p$-matrix-valued function. This facilitates the
     analysis, as for square matrices surjectivity is equivalent to
     injectivity and to invertibility. 
     
     Let again $G$ be the matrix  with entries $G_{kl} = g(k+\delta
     _k-l)$.      Since $\alpha = p/q \in \bQ$, the sequence of perturbations can
     be chosen to be $p$-periodic by Lemma~\ref{lem7}.  We now  use the
     periodicity of $\delta _k$ to bring in a variant of
     the Zibulski-Zeevi matrices~\cite{zibulski-zeevi97}.
We write  $k=r+mp$ and $l=s+np$ with $r,s  = 0, 1, \dots , p-1 $, and $m,n\in \bZ
     $. Then $\delta _k=\delta _r$ and for all  $m\in \bZ$ and $ r= 0,
     1, \dots , p-1$ we have 
     \begin{align}
       (Gc)_{k} &=      \sum _{l\in \bZ } c_l g(k+\delta _k - l)
                  \notag \\
       &=       \sum _{s=0}^{p-1} \sum _{n\in \bZ } c_{s+np} g(r+mp + \delta _r
     -s - np ) = d_{r+mp} \, . 
         \label{eq:4}
     \end{align}
For fixed $r,s$  the sum over $n$ is a convolution, so we may take
Fourier series and obtain a matrix-valued equation of functions.
Write $x_s(\xi) = \sum _{n\in \bZ } c_{s+pn} e^{2\pi i np\xi }$ for
the Fourier series of the coefficient in the residue class $s$ and  $x
(\xi ) = (x _0(\xi), \dots , x_{p-1}(\xi ))$ for the associated vector-valued
function. Likewise $y_r(\xi ) = \sum _{m\in \bZ } d_{r+pm} e^{2\pi i
  mp\xi }$ and $y(\xi ) = (y_0(\xi ) , \dots , y_{p-1}(\xi ))$.  The
Fourier series of the sequence involving $g$ is precisely the  Zak
transform of $g$,
namely,  
$$
\sum _{n\in \bZ } g(r+\delta _r -s - np) e^{2\pi i np\xi } =
Z_pg(r+\delta _r - s, \xi ) \, .
$$
Consequently, after taking Fourier series,  \eqref{eq:4}  can be
recast as
\begin{equation}
  \label{eq:5}
  \sum _{s=0}^{p-1} x_s(\xi ) Z_pg(r+\delta _r-s,\xi ) = y_r(\xi )
\qquad \text{ for }  r= 0, 1, \dots , p-1 \, .
\end{equation}
Since $c\in \ell ^1(\bZ )$ and $G$ is bounded on $\ell ^1(\bZ )$,
\eqref{eq:5} holds pointwise for all $\xi\in [0,1/p]$. 
Let $A(\xi ) $ be the $p\times p$-matrix-valued function with entries
$$
A_{rs}(\xi ) = Z_pg(r+\delta _r-s,\xi ) \qquad  r,s= 0, 1, \dots , p-1 \, .
$$
This
matrix is related to the Zeevi-Zibulski matrix that occurs in many 
criteria for Gabor frames over rational lattices~\cite{GK19,zibulski-zeevi97,zibulski-zeevi98}.  
The pointwise identity \eqref{eq:5} can be written as
\begin{equation}
  \label{eq:m1}
A(\xi ) x(\xi ) = y(\xi ) \, .
\end{equation}

We  now translate the surjectivity of  the infinite matrix $G$ to a
property of the matrix-valued function $\det A(\xi )$. 
\begin{prop} \label{inj}
Assume that $\alpha = p/q$ is rational,  $|g(t)|\leq C
(1+|t|)^{-\sigma }$ for some $C, \sigma > 1$,  and that $G$ maps
$\ell ^1(\bZ )$ \emph{onto} $\ell ^1(\bZ )$. Then $G $ is one-to-one
on $\ell ^1(\bZ )$ and thus invertible on $\ell ^1(\bZ )$. 
\end{prop}
\begin{proof}
The surjectivity of $G$ on $\ell ^1(\bZ)$ implies in particular that
 every vector  $d= (d_0, d_1, \dots , d_{p-1}) \in \bC ^p$,
i.e, $d\in \ell ^1(\bZ ), d_k = 0$ for $k<0$ and $k\geq p$ is in the range of $G$. In this
case $y_r(\xi ) = \sum _{m\in \bZ } d_{r+pm} e^{2\pi i
  mp\xi } = d_r$  
and  the vector-valued function  $y(\xi ) = d$ is constant. Due to the surjectivity
of $G$  there exists
a sequence $c=c(d)\in \ell ^1(\bZ )$ with associated Fourier series $x^{(d)}(\xi )$, such that $Gc=d$. 
In view of \eqref{eq:m1}
$$
A(\xi ) x^{(d)}(\xi ) = d \quad \quad \text{ for all } \xi \in [0,1/p]\, .
$$
Since $d\in \bC ^p$ was arbitrary, this means that for fixed $\xi $
the $p\times p$-matrix is onto $\bC ^p$.  Consequently, $A(\xi )$ is
one-to-one and thus invertible for every $\xi \in [0,1/p]$.

Now  assume that $Gc = 0$ for some  $c\in \ell ^1(\bZ )$. Then $A(\xi
) x(\xi ) = 0$ for all $\xi \in [0,1]$. Since $A(\xi )$ is invertible,
this implies that $x(\xi ) \in \bC ^p$ must be the zero vector for all
$\xi $. Consequently, $x\equiv 0$ and thus $c=0$. We have proved that
$G$ is one-to-one on $\ell ^1(\bZ )$, whence $G$ is invertible on
$\ell ^1(\bZ )$.
\end{proof}

We finally  summarize the steps that yield the proof of our main
theorem. 

\begin{proof}[Proof of Theorem~\ref{new}]
 By Lemma~\ref{lem7} we can extract  a set $\{k+\delta _k: k\in \bZ \}$
 from every $x+\alpha \bZ $, such that $\delta _k$ avoids the zero
 $x_0$ of the Zak transform $Zg$. Then by  Proposition~\ref{surj}  the matrix $G = \big(g(k+\delta _k - l)\big)
 _{k,l\in \bZ }$ is onto  $\ell ^1(\bZ
 )$. 

For rational values of $\alpha $ the matrix $G$ is also one-to-one on
$\ell ^1(\bZ )$ by Proposition~\ref{inj}.  Consequently, $G$ is invertible on $\ell
^1(\bZ )$.

Finally, the spectral invariance guaranteed by
Proposition~\ref{stab}(i) implies that $G$ is invertible on $\ell
^\infty (\bZ )$. This is precisely the sufficient condition of
Proposition~\ref{charcis} to guarantee that $\cG (g,\alpha ,1)$ is a
frame.

We are done! 
\end{proof}

Alternatively, one could give  Proposition~\ref{inj} a more structural
touch. We use
$T_p$ for the translation operator on sequences defined by $(T_pc)_k =
c_{k-p}$ and observe that the underlying matrix $G$ of
Proposition~\ref{inj} commutes with $T_p$. 
\begin{prop}
Let $G = (G_{kl})_{k,l\in \bZ }$ be a bi-infinite matrix  with off-diagonal decay
  $$
  |G_{kl}| \leq C (1+|k-l|)^{-\sigma} \qquad k,l \in \bZ \, ,
  $$
  for some $\sigma >1$. 
Assume that $G$ commutes with the translation operator $T_{p}$ for
some $p\in \bN $. If $G$ is onto $\ell ^1(\bZ )$, then $G$ is
one-to-one on $\ell^1(\bZ )$ and thus invertible. 
\end{prop}
The proof is similar to the proof of Proposition~\ref{inj}.  

\def\cprime{$'$} \def\cprime{$'$} \def\cprime{$'$} \def\cprime{$'$}
  \def\cprime{$'$} \def\cprime{$'$}

\end{document}